\documentclass[12pt,reqno]{amsart}
\usepackage{graphicx, cite, hyperref}
\usepackage{amsfonts}
\usepackage{enumerate}
\usepackage[caption = false]{subfig}
\usepackage{varioref}
\numberwithin{equation}{section}
\DeclareMathOperator{\RE}{Re}
 
 \theoremstyle{plain}

\newtheorem{thm}{Theorem}[section]
\newtheorem{cor}[thm]{Corollary}
\newtheorem{exm}[thm]{Example}
\newtheorem{lem}[thm]{Lemma}

\theoremstyle{definition}
\newtheorem{defn}[thm]{Definition}
\theoremstyle{remark}
\newtheorem{rem}[thm]{Remark}
\makeatother

\newcommand{\pe}{{\mathcal P}}
\newcommand{\de}{{\mathbb D}}
\newcommand{\ce}{{\mathbb C}}

\begin{document}

\title[]{ Mittag-Leffler Operator Connected with Certain Subclasses of Bazilevi\u{c} functions}

\author[Om Ahuja, Asena \c{C}etinkaya, Naveen Kumar Jain]{Om Ahuja, Asena \c{C}etinkaya, Naveen Kumar Jain}
\address{\textsc{Om  Ahuja}\\
	Department of Mathematical Sciences,\\
	Kent State University, Ohio, 44021, U.S.A }
	\email{oahuja@kent.edu}
		
\address{\textsc{Asena \c{C}etinkaya}\\
		Department of Mathematics and Computer Science\\
		\.{I}stanbul K\"{u}lt\"{u}r University, \.{I}stanbul, Turkey}
	\email{asnfigen@hotmail.com}
	
\address{\textsc{Naveen Kumar Jain}\\
	Department of Mathematics\\
	Aryabhatta College,  Delhi, 110021, India}
	\email{naveenjain@aryabhattacollege.ac.in}

\begin{abstract}
In this paper, we introduce  a new generalized class of analytic functions involving the  Mittag-Leffler operator and Bazilevi\u{c} functions. We examine inclusion properties,  radius problems and an application of the generalized Bernardi-Libera-Livingston integral operator for this function class.
\end{abstract}

\maketitle

\section{Introduction}
Let $\mathcal{A}$ be the family of all functions of the form
\begin{equation}\label{aa}
	f(z)=z+\sum_{n=2}^\infty a_{n}z^{n},
\end{equation}
that are analytic in the open unit disc $\mathbb{D}:=\{z : |z| < 1\}$. Denote by $\mathcal{S}$ the subfamily of  $\mathcal{A}$ consisting of functions that are univalent in $\de$. Let $\mathcal{K}$, $\mathcal{S}^\ast,$ and $\mathcal{C}$ be the well-known subclasses of $\mathcal{S}$ consisting of functions that are, respectively, convex, starlike (w.r.t. the origin) and close-to-convex in $\de$. The class $\pe$ of the functions $p$ that are analytic  in $\de$ and satisfy the conditions $p(0)=1$ and $\RE p(z)>0$ in $\de$ is also well-known in the theory of the univalent functions. For definitions, properties and history of these classes, one may refer to a survey article by the first author \cite{Ahuja} and \cite{Goodman}. Recently, Ali {\it et al.} \cite{Ali} and Anand {\it et al.} \cite{Anand} studied these classes to find various radius problems.

The Mittag-Leffler function $E_\alpha$, defined by
\begin{equation}\label{eq:mıttag}
	E_\alpha(z)=\sum_{n=0}^\infty\frac{z^n}{\Gamma(\alpha n+1)}, \quad (\alpha\in\ce; \RE (\alpha)>0; z\in\ce)
\end{equation}
was introduced in 1903 by G.M. Mittag-Leffler \cite{Mittag, Mittag2} in connection with his method of summation of some divergent series. A general form of this special function \eqref{eq:mıttag} given by
\begin{equation}\label{eq:wiman}
	E_{\alpha,\beta}(z)=\sum_{n=0}^\infty\frac{z^n}{\Gamma(\alpha n+\beta)}, \quad (\alpha, \beta\in\ce; \RE (\alpha)>0;  \RE(\beta)>0; z\in\ce)
\end{equation}
was studied by Wiman \cite{Wiman} in 1905. During this last twenty-five years, interest in Mittag-Leffler type functions \eqref{eq:mıttag} and \eqref{eq:wiman}  have significantly increased among engineers and scientists due to their applications in numerous applied problems, such as fluid flows, diffusive transport skin to diffusion, electric networks, probability, and statistically distribution theory. For detailed account of various properties and references related to applications one may refer to \cite{Agarwal, Agarwal2, Gor, Haubold, Ruzhansky}. Motivated by Sivastava and Tomovski \cite{Sri}, Attiya \cite{Attiya} and Yassen {\it et al.} \cite{Yassen} recently studied certain applications of generalized Mittag-Leffler operator involving differential subordination. Moreover, Jain {\it et al.} \cite{Ja}  studied certain unified integrals involving a multivariate Mittag-Leffler function.

Corresponding to the  function $E_{\alpha,\beta}$, Elhaddad {\it et al.} \cite{El} introduced the  Mittag-Leffler linear operator $\mathcal{E}^{m}_{\lambda,\alpha,\beta} f:\mathcal{A}\rightarrow\mathcal{A}$ given by
\begin{align}\label{eq:sri-tom}
	\mathcal{E}^{m}_{\lambda,\alpha,\beta}f(z)&=z+\sum_{n=2}^\infty\frac{\Gamma(\beta)(1+(n-1)\lambda)^m}{\Gamma(\alpha (n-1)+\beta)}a_{n}z^{n},
\end{align}
where $m\in\mathbb{N}_0=\mathbb{N}\cup\{0\}$, $\lambda\geq 0$, $\alpha, \beta\in\ce$, $\RE (\alpha)>0$, $\RE (\beta)>0$, and when $\RE(\alpha)=0$, then $\beta\neq 0$. From \eqref{eq:sri-tom}, the following recurrence formula can be easily  obtained:
\begin{equation}\label{eq:recurrence}
	\mathcal{E}^{m+1}_{\lambda,\alpha,\beta}f(z)=(1-\lambda)\mathcal{E}^{m}_{\lambda,\alpha,\beta}f(z)+\lambda z(\mathcal{E}^{m}_{\lambda,\alpha,\beta}f(z))',
\end{equation}
where $m\in\mathbb{N}_0$,  $\lambda\geq 0$. For suitable values of the parameters $m,\alpha,\beta$, and $\lambda$, we may get several linear operators; for example
\begin{enumerate}
	\item[1.] For $\alpha=0$ and $ \beta=1$, we get Al-Oboudi operator \cite{Al}.
	\item[2.] For $\alpha=0$, $ \beta=1$ and $\lambda=1$, we get S\u{a}l\u{a}gean operator \cite{Sal}.
	\item[3.] For $m=0$ and $\lambda=1$, we get the operator
	$$\mathbb{E}_{\alpha,\beta}(z)=z\Gamma(\beta)E_{\alpha,\beta}(z)=z+\sum_{n=2}^\infty\frac{\Gamma(\beta)}{\Gamma(\alpha (n-1)+\beta)}z^{n}.$$
\end{enumerate}

Let us denote by $\mathcal{B}(\vartheta,\tau,g,p)$ or briefly denote by  $\mathcal{B}$, a class of functions $f\in\mathcal{A}$ for which $p\in\pe$, $g\in\mathcal{S}^\ast$ and real numbers $\vartheta,\tau$ with $\vartheta>0$ such that
\begin{equation}\label{eq:bazilevic}
f(z)=\bigg[(\vartheta+i\tau)\int_0^zp(t)g(t)^\vartheta t^{i\tau-1}dt\bigg]^{1/(\vartheta+i\tau)},
\end{equation}
where powers are taken as principal values. Bazilevi\u{c} \cite{Bazil} proved that $\mathcal{B}\subset\mathcal{S}$. In fact, it is known that $\mathcal{K}\subset\mathcal{S}^\ast\subset\mathcal{C}\subset\mathcal{B}\subset\mathcal{S}$. For $\vartheta>0$ and $\rho<1$, we define
$$\mathcal{B}_1(\vartheta,\rho)=\bigg\{f\in\mathcal{A}: \RE\bigg(\frac{zf'(z)}{f^{1-\vartheta}(z)g^\vartheta(z)}\bigg)>\rho, z\in\de \bigg\}.$$
For $g(z)\equiv z$, we obtain
$$\mathcal{B}_2(\vartheta,\rho)=\bigg\{f\in\mathcal{A}: \RE\bigg(\frac{zf'(z)}{f(z)}\bigg(\frac{f(z)}{z}\bigg)^\vartheta\bigg)>\rho, z\in\de \bigg\},$$
$$\mathcal{B}_3(\rho):=\mathcal{B}_2(0,\rho)=\bigg\{f\in\mathcal{A}: \RE\bigg(\frac{zf'(z)}{f(z)}\bigg)>\rho, z\in\de \bigg\}=\mathcal{S}^\ast(\rho),$$
$$\mathcal{B}_4(\rho):=\mathcal{B}_2(1,\rho)=\bigg\{f\in\mathcal{A}: \RE\big(f'(z)\big)>\rho, z\in\de \bigg\}=\pe'(\rho).$$
In view of \eqref{eq:bazilevic}, Singh \cite{Singh} observed that $\mathcal{B}_1(\vartheta,0)$, $\mathcal{B}_2(\vartheta,0)$, $\mathcal{B}_3(0)=\mathcal{S}^\ast$, and $\mathcal{B}_4(0)=\pe'$ are subclasses of $\mathcal{B}$. For further details, one may refer to \cite{Bazil, Ponu}.

 In 1976, Padmanabhan \cite{Padma} introduced the class $\pe_{k}(\rho)$ of analytic functions $p$ defined in $\de$ satisfying the properties $p(0)=1$ and
\begin{equation}\label{eq.bounded}
	\int^{2\pi}_{0}\bigg|\frac{\RE p(z)-\rho}{1-\rho}\bigg|d\theta\leq k \pi,
\end{equation}
where $z=re^{i \theta}$, $k\geq 2 $ and $ 0\leq \rho < 1$.  In fact, he proved the following important result.
\begin{lem}\cite{Padma} If $p\in\pe_{k}(\rho)$, then
	$$
	p(z)=\frac{1}{2} \int^{2\pi}_{0} \frac{1+(1-2\rho)ze^{-i\theta}}{1-ze^{-i\theta}} d\mu(\theta),
	$$
	where $\mu(\theta)$ is a function with bounded variation on $[0,2\pi]$ such that
	$$
	\int^{2\pi}_{0} d\mu(\theta)=2\pi \quad\text{and}\quad \int^{2\pi}_{0} |d\mu(\theta)| \leq k \pi.
	$$
\end{lem}
From \eqref{eq.bounded}, it is observed that   $p\in \pe_{k}(\rho)$ if and only if there exists  $p_{1},p_{2}\in\pe(\rho)$ such that (see \cite{Noor})
\begin{equation}\label{e1.8}
	p(z)=\left(\frac{k}{4}+\frac{1}{2}\right)p_{1}(z)-\left(\frac{k}{4}-\frac{1}{2}\right)p_{2}(z),\quad (z\in\de).
\end{equation}

We note that for $\rho=0$, we obtain the class $\pe_{k}(0):=\pe_{k}$ defined by Pinchuk \cite{Pinchuk}. For $k=2$, we get $\pe_{2}(\rho):=\pe(\rho)$ the class of analytic functions with positive real part greater than $\rho$, and for $k=2$, $\rho=0$ we have the class $\pe_{2}(0):=\pe$ of functions with positive real part.

Now, we define the following generalized  functions class of analytic functions involving the operator $\mathcal{E}^{m}_{\lambda,\alpha,\beta}f$ given by
\eqref{eq:sri-tom}, certain subclasses of  Bazilevi\u{c} functions,   and the class $\pe_{k}(\rho)$ as:

\begin{defn}\label{def1.2} Let  $k\geq 2 $, \ $ 0\leq \rho < 1$,\ $\vartheta>0$, $m\in\mathbb{N}_0$, $\lambda\geq 0$, $\gamma\in\ce$  such that $\RE (\gamma)>0.$ Then, a function $f\in\mathcal{A}$ is in the class $\mathcal{M}^{m,\gamma}_{\lambda,\alpha,\beta}(k,\vartheta,\rho)$  if it satisfies the  condition:
	\begin{equation}\label{eq:def1}
	\bigg\{(1-\gamma)\bigg(\frac{\mathcal{E}^{m}_{\lambda,\alpha,\beta}f(z)}{z}\bigg)^\vartheta+\gamma\bigg(\frac{\mathcal{E}^{m+1}_{\lambda,\alpha,\beta}f(z)}{z}\bigg)\bigg(\frac{\mathcal{E}^{m}_{\lambda,\alpha,\beta}f(z)}{z}\bigg)^{\vartheta-1}\bigg\}\in\pe_{k}(\rho),
	\end{equation}
where $z\in\de$.
\end{defn}
\begin{rem} We observe that $\mathcal{M}^{0,1}_{1,0,1}(2,\vartheta,0)=\mathcal{B}_2(\vartheta,0)$,  $\mathcal{M}^{0,1}_{1,0,1}(2,0,0)=\mathcal{B}_3(0)$ and $\mathcal{M}^{0,1}_{1,0,1}(2,1,0)=\mathcal{B}_4(0)$ are subclasses of $\mathcal{B}$; see \cite{Bazil} and \cite{Singh}.
\end{rem}
\begin{exm} Setting $m=0$, $\lambda=1$, $\alpha=0$ and  $\beta=1$, we get
	{\small
	$$\mathcal{M}^{0,\gamma}_{1,0,1}(k,\vartheta,\rho)=:\mathcal{M}^{\gamma}(k,\vartheta,\rho)=\bigg\{f\in\mathcal{A}: (1-\gamma)\bigg(\frac{f(z)}{z}\bigg)^\vartheta+\gamma\frac{zf'(z)}{f(z)}\bigg(\frac{f(z)}{z}\bigg)^\vartheta\in\pe_k(\rho),\  0\leq \rho < 1 \bigg\} $$ }
	studied by Noor \cite{Noor}.	
\end{exm}
\begin{exm} Setting $m=0$, $\lambda=1$, $\alpha=0$, $\gamma=1$, $\beta=1$ and $k=2$, we get
$$\mathcal{M}^{0,1}_{1,0,1}(2,\vartheta,\rho)=:\mathcal{B}_2(\vartheta,\rho)=\bigg\{f\in\mathcal{A}:\frac{zf'(z)}{f(z)}\bigg(\frac{f(z)}{z}\bigg)^\vartheta\in\pe(\rho),\  0\leq \rho < 1 \bigg\} $$
defined by Bazilevi\u{c} \cite{Bazil}.	
\end{exm}
\begin{exm} Setting $m=0$, $\lambda=1$, $\alpha=0$, $\beta=1$, $\gamma=0$, $\vartheta=1$ and $k=2$, we get
	$$\mathcal{M}^{0,0}_{1,0,1}(2,1,\rho)=:\mathcal{M}(\rho)=\bigg\{f\in\mathcal{A}:\frac{f(z)}{z}\in\pe(\rho),\  0\leq \rho < 1 \bigg\} $$
	studied by Chen \cite{Chen}.	
\end{exm}

In the present investigation, using the Mittag-Leffler operator and Bazilevi˘c functions we introduce a new generalized class of analytic functions. We investigate inclusion properties, radius problems and an application of the generalized Bernardi-Libera-Livingston
integral operator for this function class. We shall also make connections with some  earlier works.

We list some preliminary lemmas required for proving our main results.
\begin{lem}\label{lem1}\cite{Miller} Let $u=u_1+iu_2$ and $v=v_1+iv_2$,  and suppose  $\Psi(u,v)$ is a complex function satisfying the conditions
\begin{enumerate}
\item[i)] $\Psi(u,v)$ is  continuous in a domain $D\subset \ce^2$,
\item[ii)] $(1,0)\in D$  and $\RE\Psi(1,0)>0$,
\item[iii)] $\RE\Psi(iu_2,v_1)\leq0$, whenever $(iu_2,v_1)\in D$ and $v_1\leq -\frac{1+u_2^2}{2}$.
\end{enumerate}

If $p(z)=1+c_1z+c_2z^2+...$ is an analytic function	in $\de$ such that $(p(z),zp'(z))\in D$ and $\RE\Psi(p(z),zp'(z))>0$ for $z\in \de$, then $\RE p(z)>0$ in $\de$.
\end{lem}
\begin{lem}\label{lem2}\cite{Ponussamy} If $p$ is  in $\pe$, and if $\zeta$ is a complex number satisfying $\RE(\zeta)\geq0,\ \zeta\neq0$, then $\RE\{p(z)+\zeta zp'(z)\}>\rho\ \ (0\leq\rho<1)$ implies that
$$\RE p(z)>\rho+(1-\rho)(2\iota_1-1),$$
where
$$\iota_1=\int_0^1\big(1+t^{\RE(\zeta)}\big)^{-1}dt,$$
and where $\iota_1$ is an increasing function of  $\RE(\zeta)$ and $1/2\leq\iota_1<1$. This estimate cannot be improved in general.
\end{lem}

\section{Inclusion Properties}
In this section, we examine some inclusion properties for the class $\mathcal{M}^{m,\gamma}_{\lambda,\alpha,\beta}(k,\vartheta,\rho)$.
\begin{thm}\label{teo1} Let $\gamma>0$ and  $f\in \mathcal{M}^{m,\gamma}_{\lambda,\alpha,\beta}(k,\vartheta,\rho)$. Then
\begin{equation}\label{eq: vareta}
\bigg(\frac{\mathcal{E}^{m}_{\lambda,\alpha,\beta}f(z)}{z}\bigg)^\vartheta\in\pe_k(\rho_1),
\end{equation}
where $\rho_1$ is given by
\begin{equation}\label{eq:rho1}
\rho_1=\frac{2\vartheta\rho+\lambda\gamma}{2\vartheta+\lambda\gamma}.
\end{equation}
\end{thm}
\begin{proof} Let
\begin{equation}\label{eq:foverg}
\bigg(\frac{\mathcal{E}^{m}_{\lambda,\alpha,\beta}f(z)}{z}\bigg)^\vartheta=(1-\rho_1)p(z)+\rho_1,
\end{equation}
where $p(0)=1$ and
$$p(z)=\left(\frac{k}{4}+\frac{1}{2}\right)p_{1}(z)-\left(\frac{k}{4}-\frac{1}{2}\right)p_{2}(z).$$
Thus,  we obtain
\begin{align}\label{align}	
& \bigg\{(1-\gamma)\bigg(\frac{\mathcal{E}^{m}_{\lambda,\alpha,\beta}f(z)}{z}\bigg)^\vartheta+\gamma\bigg(\frac{\mathcal{E}^{m+1}_{\lambda,\alpha,\beta}f(z)}{z}\bigg)\bigg(\frac{\mathcal{E}^{m}_{\lambda,\alpha,\beta}f(z)}{z}\bigg)^{\vartheta-1}\bigg\}\nonumber\\
& =(1-\gamma)\big((1-\rho_1)p(z)+\rho_1\big)+\gamma\bigg[\frac{\mathcal{E}^{m+1}_{\lambda,\alpha,\beta}f(z)}{\mathcal{E}^{m}_{\lambda,\alpha,\beta}f(z)}\big((1-\rho_1)p(z)+\rho_1\big)\bigg].
\end{align}
Taking logarithmic differentiation of \eqref{eq:foverg}, we get
$$\vartheta\bigg[\frac{z(\mathcal{E}^{m}_{\lambda,\alpha,\beta}f(z))'}{\mathcal{E}^{m}_{\lambda,\alpha,\beta}f(z)}-1\bigg]=\frac{(1-\rho_1)zp'(z)}{(1-\rho_1)p(z)+\rho_1}.$$
By using the identity  \eqref{eq:recurrence} in the last expression, we obtain
\begin{equation}\label{eq:egamma1}
\frac{\mathcal{E}^{m+1}_{\lambda,\alpha,\beta}f(z)}{\mathcal{E}^{m}_{\lambda,\alpha,\beta}f(z)}=\frac{\lambda(1-\rho_1)zp'(z)}{\vartheta\big((1-\rho_1)p(z)+\rho_1\big)}+1.
 \end{equation}
Substituting \eqref{eq:egamma1} into \eqref{align} and using \eqref{e1.8}, we arrive at
\begin{align*}	
&(1-\rho_1)p(z)+\rho_1+\frac{\lambda\gamma(1-\rho_1)zp'(z)}{\vartheta}\nonumber\\
&\ \ =\left(\frac{k}{4}+\frac{1}{2}\right)\bigg\{(1-\rho_1)p_1(z)+\rho_1+\frac{\lambda\gamma(1-\rho_1)zp_1'(z)}{\vartheta}\bigg\}\nonumber\\
&\  \ -\left(\frac{k}{4}-\frac{1}{2}\right)\bigg\{(1-\rho_1)p_2(z)+\rho_1+\frac{\lambda\gamma(1-\rho_1)zp_2'(z)}{\vartheta}\bigg\}.
\end{align*}
Since $f\in \mathcal{M}^{m,\gamma}_{\lambda,\alpha,\beta}(k,\vartheta,\rho)$, it follows that
$$\bigg\{(1-\rho_1)p_i(z)+\rho_1+\frac{\lambda\gamma(1-\rho_1)zp_i'(z)}{\vartheta}\bigg\}\in\pe(\rho), \ \ (0\leq\rho<1;  i=1,2).$$
That is
\begin{equation}\label{align2}
\frac{1}{1-\rho}\bigg\{(1-\rho_1)p_i(z)+\rho_1+\frac{\lambda\gamma(1-\rho_1)zp_i'(z)}{\vartheta}-\rho\bigg\}\in\pe.
\end{equation}
To prove the theorem, we will show that $p_i\in\pe, \ (i=1,2)$. We form the functional $\Psi(u,v)$ by taking $u=u_1+iu_2$ and $v=v_1+iv_2$ such that
$$\Psi(u,v)=(1-\rho_1)u+\rho_1-\rho+\frac{\lambda\gamma(1-\rho_1)v}{\vartheta}.$$
Using \eqref{align2}, it is easy to  see that the first two conditions of Lemma \ref{lem1} are satisfied. To verify the condition (iii),  we get
\begin{align}\label{align3}
\RE\Psi(iu_2,v_1)&=\rho_1-\rho+\RE\bigg\{\frac{\lambda\gamma(1-\rho_1)v_1}{\vartheta}\bigg\}\nonumber\\
&\leq\rho_1-\rho-\frac{\lambda\gamma(1-\rho_1)(1+u_2^2)}{2\vartheta}\nonumber\\
&=\frac{2\vartheta\rho_1-2\vartheta\rho-\lambda\gamma(1-\rho_1)(1+u_2^2)}{2\vartheta}\nonumber\\
&=\frac{2\vartheta(\rho_1-\rho)-\lambda\gamma(1-\rho_1)-\lambda\gamma(1-\rho_1)u^2_2}{2\vartheta}:=\frac{A+Bu^2_2}{2C},\nonumber
\end{align}
where $v_1\leq-\frac{1+u_2^2}{2}$. Now, $\RE\Psi(iu_2,v_1)\leq 0$ if
\begin{align}
& A=2\vartheta(\rho_1-\rho)-\lambda\gamma(1-\rho_1)\leq0,\nonumber\\
& B=-\lambda\gamma(1-\rho_1)\leq0,\nonumber\\
& C=\vartheta>0.\nonumber
\end{align}
 From $A\leq0$, we obtain $\rho_1$ as given by \eqref{eq:rho1} and $B\leq0$ ensures that $0\leq\rho_1<1$.  In view of  Lemma \ref{lem1}, for $p(z)=p_i(z)$, $p_i\in\pe$, where $i=1, 2$. Consequently, $p\in\pe_k(\rho_1)$.
\end{proof}
For $m=0$, $\lambda=1$, $\alpha=0$ and  $\beta=1$, Theorem \ref{teo1} reduces to the  following  new result.
\begin{cor} Let $\gamma>0$ and  $f\in \mathcal{M}^{\gamma}(k,\vartheta,\rho)$. Then
$\big(\frac{f(z)}{z}\big)^\vartheta\in\pe_k(\rho_1)$, where $\rho_1$ is given by
$$
\rho_1=\frac{2\vartheta\rho+\gamma}{2\vartheta+\gamma}.
$$
\end{cor}
By using the inclusion relation given in Theorem \ref{teo1}, we prove the following theorem.
\begin{thm}\label{teo2} Let \ $\vartheta>0$ and $0\leq \gamma_1<\gamma_2$. Then $\mathcal{M}^{m,\gamma_2}_{\lambda,\alpha,\beta}(k,\vartheta,\rho)\subset\mathcal{M}^{m,\gamma_1}_{\lambda,\alpha,\beta}(k,\vartheta,\rho).$
\end{thm}
\begin{proof} Let $f\in\mathcal{M}^{m,\gamma_2}_{\lambda,\alpha,\beta}(k,\vartheta,\rho)$. Then, we have
$$H_2(z)=\bigg\{(1-\gamma_2)\bigg(\frac{\mathcal{E}^{m}_{\lambda,\alpha,\beta}f(z)}{z}\bigg)^\vartheta+\gamma_2\bigg(\frac{\mathcal{E}^{m+1}_{\lambda,\alpha,\beta}f(z)}{z}\bigg)\bigg(\frac{\mathcal{E}^{m}_{\lambda,\alpha,\beta}f(z)}{z}\bigg)^{\vartheta-1}\bigg\}\in\pe_{k}(\rho).$$
In view of Theorem \ref{teo1}, we conclude that
$$\bigg(\frac{\mathcal{E}^{m}_{\lambda,\alpha,\beta}f(z)}{z}\bigg)^\vartheta:=H_1(z)\in\pe_k(\rho_1)\subset\pe_k(\rho).$$	
Thus, for $\gamma_1\geq0$,
\begin{equation}\label{convex}
\begin{split}	\bigg\{(1-\gamma_1)\bigg(\frac{\mathcal{E}^{m}_{\lambda,\alpha,\beta}f(z)}{z}\bigg)^\vartheta+\gamma_1\bigg(\frac{\mathcal{E}^{m+1}_{\lambda,\alpha,\beta}f(z)}{z}\bigg)\bigg(\frac{\mathcal{E}^{m}_{\lambda,\alpha,\beta}f(z)}{z}\bigg)^{\vartheta-1}\bigg\}\\
  =\bigg(1-\frac{\gamma_1}{\gamma_2}\bigg)H_1(z)+\frac{\gamma_1}{\gamma_2}H_2(z).
 \end{split}
\end{equation}
Because the class $\pe_k(\rho)$ is a convex set (see \cite{Noor}), it follows that the right sides of \eqref{convex} belongs to $\pe_k(\rho)$, therefore $f\in\mathcal{M}^{m,\gamma_1}_{\lambda,\alpha,\beta}(k,\vartheta,\rho)$.
\end{proof}
For $m=0$, $\lambda=1$, $\alpha=0$, and  $\beta=1$, Theorem \ref{teo2} reduces to the  following  inclusion result.
\begin{rem}\cite{Noor} Let $\vartheta>0$ and $0\leq \gamma_1<\gamma_2$. Then $\mathcal{M}^{\gamma_2}(k,\vartheta,\rho)\subset\mathcal{M}^{\gamma_1}(k,\vartheta,\rho).$
\end{rem}

\section{Radius Problem}
In this section, we examine certain radius problems.
\begin{thm}\label{teo3} If a  function $f\in\mathcal{A}$ satisfies
\begin{equation}\label{eq:radii}
\bigg(\frac{\mathcal{E}^{m}_{\lambda,\alpha,\beta}f(z)}{z}\bigg)^\vartheta\in\pe_k(\rho),
\end{equation}
then $f\in\mathcal{M}^{m,\gamma}_{\lambda,\alpha,\beta}(k,\vartheta,\rho)$ for $|z|< r_1$, where
\begin{equation}\label{eq:radius}
r_1=\frac{\lambda\gamma+\vartheta-\sqrt{\lambda^2\gamma^2+2\lambda\gamma\vartheta}}{\vartheta}.
\end{equation}
\end{thm}
\begin{proof} In view of \eqref{eq:radii},  we have
\begin{equation}\label{eq:pk}
\bigg(\frac{\mathcal{E}^{m}_{\lambda,\alpha,\beta}f(z)}{z}\bigg)^\vartheta=(1-\rho)p(z)+\rho,
\end{equation}
where $p\in\pe_k$.

Hence, by using \eqref{eq:recurrence} and  \eqref{eq:pk}, we easily get
\begin{align}\label{eq:radius2}
&\frac{1}{1-\rho}\bigg\{(1-\gamma)\bigg(\frac{\mathcal{E}^{m}_{\lambda,\alpha,\beta}f(z)}{z}\bigg)^\vartheta+\gamma\bigg(\frac{\mathcal{E}^{m+1}_{\lambda,\alpha,\beta}f(z)}{z}\bigg)\bigg(\frac{\mathcal{E}^{m}_{\lambda,\alpha,\beta}f(z)}{z}\bigg)^{\vartheta-1}-\rho\bigg\}\nonumber\\	
&=p(z)+\frac{\lambda\gamma zp'(z)}{\vartheta}\nonumber\\
&=\left(\frac{k}{4}+\frac{1}{2}\right)\bigg\{p_1(z)+\frac{\lambda\gamma zp_1'(z)}{\vartheta}\bigg\} -\left(\frac{k}{4}-\frac{1}{2}\right)\bigg\{p_2(z)+\frac{\lambda\gamma zp_2'(z)}{\vartheta}\bigg\},
\end{align}	
where $p_1,p_2\in \pe$ and $z\in\de$.\\

Now, by using well-known estimates (see \cite{Goodman}) for the class $\pe$,
$$|zp_i'(z)|\leq \frac{2r\RE p_i(z)}{1-r^2},\quad\quad\quad\quad\quad\quad\quad\quad\quad\quad\quad\quad\quad\quad\quad\quad\quad\quad\quad\quad\quad$$
$$\RE p_i(z)\geq\frac{1-r}{1+r},\ \ (|z|<r<1; i=1,2; z\in\de),\quad\quad\quad\quad\quad\quad\quad\quad\quad\quad$$	
we have
\begin{align}	
\RE\bigg\{p_i(z)+\frac{\lambda \gamma zp'_i(z)}{\vartheta}\bigg\}&\geq\RE\bigg\{p_i(z)-\frac{\lambda \gamma |zp'_i(z)|}{\vartheta}\bigg\}\nonumber\\
&\geq\RE p_i(z)\bigg\{1-\frac{2\lambda \gamma r}{\vartheta(1-r^2)}\bigg\}\nonumber\\
&=\RE p_i(z)\bigg\{\frac{\vartheta(1-r)^2-2\lambda\gamma r}{\vartheta(1-r)^2}\bigg\}.\nonumber
\end{align}		
The right hand 	side of the last inequality is positive if $r<r_1$, where $r_1$ is given by \eqref{eq:radius}. The result is sharp, therefore from \eqref{eq:radius2} we prove that $f\in\mathcal{M}^{m,\gamma}_{\lambda,\alpha,\beta}(k,\vartheta,\rho)$ for $|z|< r_1$.

The sharp function is proven in the following. From \eqref{eq:radii}, we have
$$\bigg(\frac{\mathcal{E}^{m}_{\lambda,\alpha,\beta}f(z)}{z}\bigg)^\vartheta=p(z)=\left(\frac{k}{4}+\frac{1}{2}\right)p_{1}(z)-\left(\frac{k}{4}-\frac{1}{2}\right)p_{2}(z)\in\pe_k(\rho),$$
where $p_i(z)=\frac{1+(1-2\rho)z}{1-z}\ (i=1,2)$, then clearly we  write
\begin{equation}\label{eq:sharp}
(\mathcal{E}^{m}_{\lambda,\alpha,\beta}f(z))^\vartheta=z^\vartheta p(z).
\end{equation}
Differentiating on both sides of the equation \eqref{eq:sharp}, and using  \eqref{eq:recurrence} together with routine calculations, we easily get
$$
(1-\gamma)\bigg(\frac{\mathcal{E}^{m}_{\lambda,\alpha,\beta}f(z)}{z}\bigg)^\vartheta+\gamma\bigg(\frac{\mathcal{E}^{m+1}_{\lambda,\alpha,\beta}f(z)}{z}\bigg)\bigg(\frac{\mathcal{E}^{m}_{\lambda,\alpha,\beta}f(z)}{z}\bigg)^{\vartheta-1}=p(z)+\frac{\lambda\gamma zp'(z)}{\vartheta}\in\pe_k(\rho).
$$
Thus, by applying Hallenbeck and Ruscheweyh \cite{Rusc}, we observe that
\begin{multline}\label{eq:sharp2}
	\bigg(\frac{\mathcal{E}^{m}_{\lambda,\alpha,\beta}f(z)}{z}\bigg)^\vartheta=\frac{\vartheta}{\lambda\gamma}\int_0^1\bigg[\left(\frac{k}{4}+\frac{1}{2}\right)u^{\frac{\vartheta}{\lambda\gamma}-1}\frac{1+(1-2\rho)uz}{1-uz}\\
	-\left(\frac{k}{4}-\frac{1}{2}\right)u^{\frac{\vartheta}{\lambda\gamma}-1}\frac{1-(1-2\rho)uz}{1+uz}\bigg]du
\end{multline}
is the sharp function for  $f\in\mathcal{M}^{m,\gamma}_{\lambda,\alpha,\beta}(k,\vartheta,\rho)$.  We note that for $m=0$, $\lambda=1$, $\alpha=0$ and  $\beta=1$, the sharp function given by \eqref{eq:sharp2} reduces to the sharp function given in \cite{Noor}.
\end{proof}
Putting $m=0$, $\lambda=1$, $\alpha=0$ and  $\beta=1$, Theorem \ref{teo3} reduces to the  following new result.
\begin{cor} If a  function $f\in\mathcal{A}$ satisfies $\big(\frac{f(z)}{z}\big)^\vartheta\in\pe_k(\rho)$, then $f\in\mathcal{M}^{\gamma}(k,\vartheta,\rho)$ for $|z|< r_2$, where
$$r_2=\frac{\gamma+\vartheta-\sqrt{\gamma^2+2\gamma\vartheta}}{\vartheta}.$$
\end{cor}
\begin{rem} Setting $m=\alpha=0$, $\gamma=\beta=\lambda=\vartheta=1$, and $k=2$ in  Theorem \ref{teo3}, then $f$ is in  $\mathcal{M}^{0,1}_{1,0,1}(2,1,\rho)=:\mathcal{B}_2^1(1,1,1,\rho,1,0)$ for $|z|<2-\sqrt{3}\approx 0.2679$, which was given  in \cite[Theorem 3.4]{Noor2}.
\end{rem}

\section{Application of An Integral Operator}
In this section, we consider an application of the generalized Mittag-Leffler operator given by \eqref{eq:sri-tom} involving  the generalized Bernardi-Libera-Livingston integral operator $\mathcal{L}_{\sigma}:\mathcal{A}\rightarrow \mathcal{A}$ given by
\begin{equation}\label{eq:bernardi}
\mathcal{L}_{\sigma}f(z)=\frac{\sigma+1}{z^\sigma}\int_0^z t^{\sigma-1}f(t)dt \ \ \ \ (\sigma>-1).
\end{equation}
From this operator, we easily get
\begin{equation}\label{eq:int}
z(\mathcal{E}^{m}_{\lambda,\alpha,\beta}\mathcal{L}_{\sigma}f(z))'=(\sigma+1)\mathcal{E}^{m}_{\lambda,\alpha,\beta}f(z)-\sigma \mathcal{E}^{m}_{\lambda,\alpha,\beta}\mathcal{L}_{\sigma}f(z).
\end{equation}
For several special cases of this operator and related operators, one may refer to a survey article by the first two authors \cite{Ahuja2} and related references therein.
\begin{thm}\label{teo4} Let  $f\in\mathcal{A}$ and $\mathcal{L}_{\sigma}f$ be given by \eqref{eq:bernardi}. If
\begin{equation}\label{eq:subo}
\bigg\{(1-\gamma)\frac{\mathcal{E}^{m}_{\lambda,\alpha,\beta}\mathcal{L}_{\sigma}f(z)}{z}+\gamma\frac{\mathcal{E}^{m}_{\lambda,\alpha,\beta}f(z)}{z}\bigg\}\in\pe_k(\rho),
\end{equation}
then
$$\frac{\mathcal{E}^{m}_{\lambda,\alpha,\beta}\mathcal{L}_{\sigma}f(z)}{z}\in\pe_k(\iota), \ \ (z\in\de),$$
where $\iota$ is given by
\begin{equation}\label{eq:iota1}
\iota=\rho+(1-\rho)(2\iota_1-1), \ \ \iota_1=\int_0^1(1+t^{\frac{\gamma}{\sigma+1}})^{-1}dt.
\end{equation}	
\end{thm}
\begin{proof} Consider the function
$$\frac{\mathcal{E}^{m}_{\lambda,\alpha,\beta}\mathcal{L}_{\sigma}f(z)}{z}=p(z)=\bigg\{\left(\frac{k}{4}+\frac{1}{2}\right)p_{1}(z)-\left(\frac{k}{4}-\frac{1}{2}\right)p_{2}(z)\bigg\}.$$	
Differentiating both sides, and using 	\eqref{eq:int} we get
$$\frac{\mathcal{E}^{m}_{\lambda,\alpha,\beta}f(z)}{z}=p(z)+\frac{zp'(z)}{\sigma+1}.$$
If we use the identity given by \eqref{eq:subo}, we obtain 	
$$(1-\gamma)\frac{\mathcal{E}^{m}_{\lambda,\alpha,\beta}\mathcal{L}_{\sigma}f(z)}{z}+\gamma\frac{\mathcal{E}^{m}_{\lambda,\alpha,\beta}f(z)}{z}=p(z)+\frac{\gamma zp'(z)}{\sigma+1}\in\pe_k(\rho).$$	
This implies that 	
$$\RE\bigg\{p_i(z)+\frac{\gamma zp'_i(z)}{\sigma+1}\bigg\}>\rho,\ \ (i=1,2).$$	
By using Lemma \ref{lem2}, we see that $\RE\{p_i(z)\}>\iota$, where $\iota$ is given by \eqref{eq:iota1}. Thus, we arrive at $p\in\pe_k(\iota)$. This completes the proof.
\end{proof}
Setting $m=0$, $\lambda=1$, $\alpha=0$, and  $\beta=1$ in operator $\mathcal{E}^{m}_{\lambda,\alpha,\beta}f$, Theorem \ref{teo4} gives  the  following result.
\begin{cor} Let  $f\in\mathcal{A}$ and $\mathcal{L}_{\sigma}f$ be given by \eqref{eq:bernardi}. If
$$
\bigg\{(1-\gamma)\frac{\mathcal{L}_{\sigma}f(z)}{z}+\gamma\frac{f(z)}{z}\bigg\}\in\pe_k(\rho),
$$
then
$\frac{\mathcal{L}_{\sigma}f(z)}{z}\in\pe_k(\iota)$, where $\iota$ is given by
$$\iota=\rho+(1-\rho)(2\iota_1-1), \ \ \iota_1=\int_0^1(1+t^{\frac{\gamma}{\sigma+1}})^{-1}dt.$$	
\end{cor}
\section*{\bf{Conclusion}}
We conclude our investigation by remarking that the defined new generalized class of analytic functions involving the  Mittag-Leffler operator and Bazilevi\u{c} functions gives various well known subclasses of Bazilevi\u{c} functions as particular cases which in turn yields many proved results as corollary. We are investigating the main results to find potentially useful applications in a variety of areas.

\section*{\bf{Data Availability}}

No data were used to support this study.

\section*{\bf{Conflicts of Interest}}

The authors declare that they have no conflicts of interest.

\section*{\bf{Acknowledgment}}

 The authors are grateful to the referee for the helpful suggestions and insights that helped improve the clarity of this manuscript.

\end{document}